\documentclass[12pt]{amsproc}
\usepackage{geometry}             
\usepackage{float}
\usepackage{bm}
\usepackage{fullpage,xcolor}
\usepackage{mathabx}
\usepackage[mathscr]{euscript}
\usepackage[all]{xy}
\usepackage{epsfig}
\usepackage[T1]{fontenc}
\usepackage{amsfonts}
\usepackage{tikz-cd}

\usepackage{graphicx}
\usepackage{amssymb}
\usepackage{amsmath}
\usepackage{amsthm}
\usepackage{mathrsfs}
\usepackage{epstopdf}
\usepackage{url}
\usepackage{hyperref}
\usepackage[msc-links,alphabetic]{amsrefs}

\usepackage{upgreek}

\usepackage{color}
\usepackage{subfig} 

\newcommand{\ep}{\varepsilon}

\newcommand{\la}{\lambda}
\renewcommand{\phi}{\varphi}

\newcommand{\Si}{\Sigma}
\newcommand{\ZZ}{{\mathbb Z}}

\newcommand{\RR}{{\mathbb R}}

\newcommand{\cE}{\mathcal E}

\newcommand{\Aut}{\operatorname{Aut}}
\newcommand{\Out}{\operatorname{Out}}
\newcommand{\Inn}{\operatorname{Inn}}
\newcommand{\Homeo}{\operatorname{Homeo}}

\newcommand{\Int}{\operatorname{Int}}
\newcommand{\Isom}{\operatorname{Isom}}
\newcommand{\Mcg}{\operatorname{Mcg}}

\newcommand{\sm}{\smallsetminus}

\newcommand{\id}{{\text{id}}}

\newcommand*\wbar[1]{
  \hbox{ \kern-0.3em%
    \vbox{%
      \hrule height 0.5pt  
      \kern0.25ex
      \hbox{%
        \kern-0.15em
        \ensuremath{#1}%
        \kern-0.05em
      }%
    }%
  \kern0.05em}%
}

\newtheorem{theorem}{Theorem} 

\newtheorem{proposition}[theorem]{Proposition}

\newtheorem{corollary}[theorem]{Corollary}

\theoremstyle{definition}     
\newtheorem{definition}[theorem]{Definition}

\theoremstyle{remark}
\newtheorem{remark}[theorem]{Remark}

\newtheorem{question}[theorem]{Question}

\title{Virtual and welded periods of classical knots}
\author[H. U. Boden]{Hans U. Boden}
\address{Mathematics \& Statistics, McMaster University, Hamilton, Ontario}
\email{boden@mcmaster.ca}
\urladdr{math.mcmaster.ca/~boden}

\author[A. J. Nicas]{Andrew J. Nicas}
\address{Mathematics \& Statistics, McMaster University, Hamilton, Ontario}
\email{nicas@mcmaster.ca}


\date{\today}                                           

\begin{document}

\begin{abstract}
We show that any virtual or welded period of a classical knot can be realized as a classical period. A direct consequence is that a classical knot admits only finitely many virtual or welded periods. 
\end{abstract}

\maketitle


In \cite{Fox-1962-a}, Fox gave an overview of the   Smith Conjecture,  which was unresolved at the time, concerning periodic transformations of the $3$-sphere preserving a simple closed curve.
\hbox{Question 7}  
of \cite{Fox-1962-a} deals with the periodic symmetries of a knot, which he observed come in eight different varieties.  
Of special interest are  the two cases  
corresponding to knots admitting a ``free period'' or a ``cyclic period'', respectively.
A \emph{symmetry} of a (tame) knot $K\subset S^3$ is a diffeomorphism $\varphi \colon S^3 \to S^3$ of finite order $p$ (the \emph{period} of $\varphi$)
such that $\varphi(K)=K$.
A knot $K$ is said to be:
\begin{itemize}
\item \emph{freely periodic} if there is a fixed point free, orientation preserving symmetry of $K$,
\item  \emph{(cyclically) periodic} if there is an orientation preserving symmetry of $K$ whose fixed point set is non-empty and disjoint from $K$.
\end{itemize}
In the second case,
the solution of the Smith Conjecture \cite{Bass-Morgan} implies that the fixed point set is an unknotted circle and that the symmetry is conjugate in the diffeomorphism group of $S^3$
to a rotation about an axis in $S^3$.
A fundamental result of E.~Flapan asserts that a nontrivial classical knot admits only finitely many periods (see \cite{Flapan-1985}, where she also shows that any knot, other than a torus knot, admits only finitely many free periods).

Virtual knots and links were defined by Kauffman in \cite{Kauffman-1999} as equivalence classes of virtual knot diagrams by the extended Reidemeister moves. Very briefly, a virtual knot diagram is a 4-valent graph in the plane with two types of crossings, classical and virtual crossings. Classical crossings are depicted with over and under-crossings as usual, and virtual crossings are circled.
A diagram without virtual crossings is called classical because it represents a classical knot or link.
In \cite{Goussarov-Polyak-Viro}, it is shown that for classical knots, equivalence through virtual knot diagrams implies equivalence as classical knots, and thus virtual knot theory is an extension of classical knot theory.

Allowing the first forbidden move (see  \cite{Kauffman-1999}) leads to the coarser notion of \emph{welded knots}.
Satoh showed that welded knots represent ribbon knotted tori in $S^4$ \cite{Satoh-2000}.
Two classical knots that are equivalent as welded knots are equivalent as classical knots, and so classical knot theory embeds into the theory of welded knots.
In this paper, we are concerned with symmetries of classical knots that arise from periodic virtual knot diagrams, which were first
introduced by S.\,Y.~Lee in \cite{Lee-2012} and which are extended here to include welded knots.

\begin{definition}
Let $p$ be a positive integer.
A virtual knot diagram that misses the origin is said to be \emph{$p$-periodic} if it is invariant under a rotation about the origin in the plane by an angle of $2\pi/p$.
A virtual or welded knot $K$ is said to be \emph{$p$-periodic} if it admits a $p$-periodic virtual knot diagram, and in that case, we say that $p$ is a \emph{virtual period} or \emph{welded period} for the virtual or welded knot $K$.
\end{definition}

S.\,Y.~Lee asked whether classical knots or links can admit ``exotic'' virtual periods, cf.~\cite[Question 4]{Lee-2012}, and we rephrase his question to include welded knots as follows.

\begin{question}
Can a classical knot admit virtual or welded periods which are different from its classical periods?  
\end{question}

The following theorem answers this question and is our main result.

\begin{theorem} \label{thm-main}
If a classical knot admits a $p$-periodic virtual or welded knot diagram, then it admits a $p$-periodic classical knot diagram. The sets of virtual and welded periods of a given classical knot are the same as its set of classical (cyclic) periods. In particular, a nontrivial classical knot admits only finitely many virtual and welded periods.
\end{theorem}

Let $N(K)$ be a regular neighborhood of a knot $K \subset S^3$.
The \emph{exterior} of $K$ is the compact manifold with torus boundary: $E_K = S^3 \sm \Int(N(K))$.
Let $\Mcg(E_K)$ denote the \emph{mapping class group of $E_K$}, that is, the group of diffeomorphisms of $E_K$ modulo the subgroup of diffeomorphisms isotopic to the identity.
Let  $G_K=\pi_1(E_K)$  denote knot group of $K$.

The proof of Theorem \ref{thm-main} will be presented in \ref{sec:proof}.
We sketch the argument in the special case where  $K$ is a hyperbolic knot, that is, $\Int(E_K)$ admits a complete hyperbolic metric of finite volume.
Let $\Out(G_K)$ be the group of  outer automorphisms of $G_K$ (see \ref{section:autom}) and
let $\Isom(E_K)$ be the group of isometries of the complete hyperbolic metric of finite volume.

\begin{theorem}[see Theorem 10.5.3 \cite{Kawauchi-1990}] \label{thm-Mostow}
If $K$ is a hyperbolic knot, then there are isomorphisms
$~\Mcg(E_K) \cong \Out(G_K) \cong \Isom(E_K)$.
\end{theorem}

In Theorem \ref{thm-Mostow}, the isomorphism $\Mcg(E_K) \cong \Out(G_K)$ is a special case of \cite[Corollary 7.5]{Waldhausen-1968}, and the isomorphism 
$\Out(G_K) \cong \Isom(E_K)$ makes use of Mostow's rigidity theorem \cite{Mostow},
as extended to complete, finite volume hyperbolic manifolds by Prasad \cite{Prasad}.

Assume that $K$ admits a $p$-periodic virtual knot diagram.
Such a diagram gives rise to an automorphism $\phi \colon G_K \to G_K$ which, by Proposition \ref{prop-order}, must have order $p$.
The fundamental group of any complete, finite volume hyperbolic manifold is torsion free and has trivial center
and so $G_K$, when $K$ is a hyperbolic knot, has these properties.
This observation allows us to conclude that the image of $\phi$ in $\Out(G_K)$ also has order $p$, and Theorem \ref{thm-Mostow} applies to show that
$\phi$ corresponds to an isometry
 $\phi_* \in \Isom(E_K)$ of order $p$.
There is a flat torus $T$ obtained as a cross section of the cusp associated to the hyperbolic structure on $\Int(E_K)$ and a compact submanifold $E'_K \subset \Int(E_K)$,  with $T = \partial E'_K$,
that is also an exterior for $K$ and is invariant under $\varphi_*$.
By \cite[Theorem 2]{Luo-1992},  $(\varphi_*)|_{E'_K}$ extends to an order $p$ diffeomorphism  ${\widehat \varphi}_ *\colon S^3 \rightarrow S^3$ preserving $K$.
(Alternatively, one can conclude this from the main theorem of \cite{Gordon-Luecke-1989b} as follows. The diffeomorphism $(\varphi_*)|_{E'_K} \colon E'_K \to E'_K$  maps a meridian of $K$ to another meridian for $K$, and 
any $p$-periodic diffeomorphism defined on the boundary torus $\partial V$ of a solid torus $V$ which sends a meridian to another meridian  can be extended to a $p$-periodic diffeomorphism of $V$.)
Furthermore, Proposition \ref{good-guys-win} can be used to show  ${\widehat \varphi}_ *$ is orientation preserving.
While this argument produces an orientation preserving symmetry of $K$ of order $p$, it does not immediately exclude the possibility that ${\widehat \varphi}_ *$ could be a free symmetry.
That exclusion is accomplished using our
Theorem \ref{thm:fixedpoint} as in the proof of Theorem \ref{thm-main} given in \ref{sec:proof}.

\setcounter{section}{1} \noindent
\subsection{Periodic virtual knots and Wirtinger presentations}\label{sec:periodic-virtual}
Let $K$ be a virtual knot, and let $G_K$ denote its knot group. Given a virtual knot diagram for $K$, we obtain a Wirtinger presentation for $G_K$, and the isomorphism class of $G_K$ and its peripheral structure $P_K$ are invariants of the underlying welded equivalence class of $K$ \cite{Kim-2000}. If $K$ is classical, then $G_K = \pi_1(E_K)$, the fundamental group of the exterior of the knot $E_K = S^3 \sm \Int(N(K))$, and $P_K$ is the conjugacy class of the image of the fundamental group of $\partial E_K$ in $G_K$.

Suppose $K$ is a classical or virtual knot and $D$ is a $p$-periodic virtual knot diagram representing $K$.
The periodic transformation of the diagram induces an automorphism $\phi \colon G_K \to G_K$ as follows. Recall that virtual knots can be represented as knots in thickened surfaces $\Si \times I$. From this point of view, the knot group is the fundamental group of the complement of $K$ in the cone space $C(\Si \times I)=\Si \times I/\!\!\sim$ obtained by identifying $(x,1) \sim (y,1)$ for all $x,y \in \Sigma$. Notice that the periodic transformation of the diagram defines a homeomorphism of $C(\Si \times I)$ fixing the cone point, which we take to be the base point in the fundamental group.

Using symmetry, it follows that $D$ must have $np$ real crossings and also $np$ arcs, for some $n >0$, which are paths from one real undercrossing to the next undercrossing, ignoring virtual crossings.
We denote the quotient knot by $K_\coAsterisk$, it is represented by the diagram $D_\coAsterisk$ obtained by taking the quotient of $D$ by the $\ZZ/p$ action.
Notice that $D_\coAsterisk$ has $n$ real crossings.

We will use the symmetry of the diagram to produce a particularly useful Wirtinger presentation for the knot group $G_K$, following the approach given in \cite[Ch. 14D]{BZH-2014}, cf. \cite[\S 3.1]{Boden-Nicas-White-2017}. 
Let $$c_{1,0}, \ldots, c_{n,0}, c_{1,1}, \ldots, c_{n,1}, 
\ldots, c_{1,p-1}, \ldots, c_{n,p-1}$$  
denote the undercrossings of $D$ in the order in which they are encountered as one travels around the knot, and 
let  $\ep_{i,j}$ denote the writhe of $c_{i,j}$ for $1\leq i \leq n$ and $0 \leq j \leq p-1.$
Notice that by symmetry, $\ep_{i,j} = \ep_{i,j+1}$ for $0 \leq j \leq p-1,$ so we will write $\ep_i$ instead of $\ep_{i,j}.$
We choose the generator $\phi$ of $\ZZ/p$ that acts on the diagram $D$ by sending $c_{i,j}$ to $c_{i,j+1}$, where $1\leq i\leq n$ and $0\leq j \leq p-1$ and $j+1$ is taken mod $p$ (i.e., if $j=p-1,$ then $j+1$ is taken to equal $0$).

Using this setup, we can label the arcs of the diagram $D$ 
$$a_{1,0}, \ldots, a_{n,0}, a_{1,1}, \ldots, a_{n,1}, 
\ldots, a_{1,p-1}, \ldots, a_{n,p-1}$$ so that the arc $a_{i,j}$ starts at the crossing $c_{i-1,j}$ and ends at the crossing $c_{i,j}$ for $i=2, \ldots, n$. (In case $i=1$, the arc $a_{1,j}$ starts at $c_{n,j-1}$ and ends at $c_{1,j}$.) The automorphism acts on the arcs by sending $a_{i,j}$ to $a_{i,j+1}$ for 
$1\leq i\leq n$ and $0\leq j \leq p-1$ with $j$ taken mod $p$ as above.  

For example, Figure \ref{example} depicts a 3-periodic diagram of a virtual knot with crossings $c_{i,j}$ and arcs $a_{i,j}$ labelled according to our conventions. Under clockwise rotation of $2 \pi/3$ about the center, the automorphism $\phi$ acts by sending the crossing $c_{i,j}$ to $c_{i,j+1}$ 
and the arc $a_{i,j}$ to $a_{i,j+1}$ for $i=1,2$ and $j=0,1,2$ with $j+1$ taken mod 3.

\begin{figure}[ht]
\includegraphics[scale=0.75]{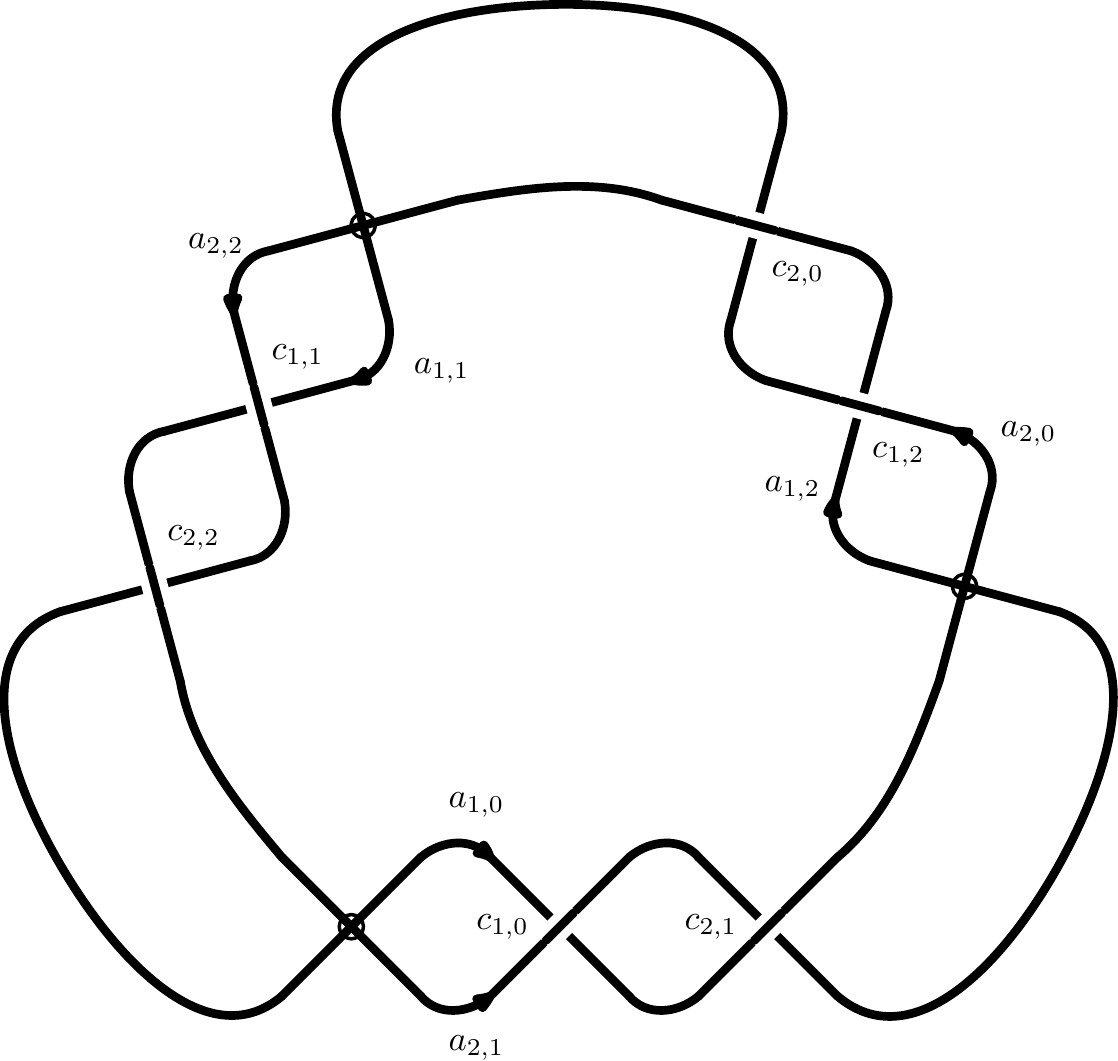}  
\caption{{A 3-periodic virtual knot}} \label{example}
\end{figure}
 
In general, for a periodic diagram, the Wirtinger relation of the crossing $c_{i,j}$ is given by
$$r_{i,j} = (a_{k,\ell})^{-\ep_{i}} \, a_{i,j}^{ } \, (a_{k,\ell})^{\ep_{i}} \, (a_{i+1,j})^{-1} $$
for $i=1,\ldots, n-1$ and $j=0,\ldots, p-1.$
When $i=n$, we get the relation $r_{n,j} = (a_{k,\ell})^{-\ep_{n}} \, a_{n,j}^{ } \,(a_{k,\ell})^{\ep_{n}}(a_{1,j})^{-1}.$

The resulting Wirtinger presentation of the $p$-periodic virtual knot $K$ is then
\begin{equation} \label{Wirtinger-pres-K}
G_K = \langle a_{i,j}  \mid r_{i,j} \rangle,
\end{equation} 
where $1 \leq i \leq n, \; 
0 \leq j \leq p-1$.
This presentation admits a $\ZZ/p$ symmetry, and the Wirtinger presentation for the knot group of the quotient knot $K_\coAsterisk$  is obtained by adding the relations $a_{i,0} = a_{i,1} = \cdots = a_{i,p-1}$ for $1 \leq i \leq n$, which gives the presentation for the quotient group
$$G_{K_\coAsterisk} = \langle a_1, \ldots, a_n \mid r_{1}, \ldots, r_n \rangle,$$
where $a_i$ refers to the equivalence class $\{a_{i,0}, \ldots, a_{i,p-1}\}$ of generators and $r_i$ is the relation $(a_{k})^{-\ep_{i}} \,a_{i}^{}\, (a_{k})^{\ep_{i}} (a_{i+1})^{-1} $ with $i+1$ taken mod $n$.

Let $\pi \colon G_K \to G_{K_\coAsterisk}$  denote the quotient map.

For an abstract group $G$, we use $\Aut(G)$ to denote its group of automorphisms.

\begin{proposition} \label{prop-order}
Suppose $K$ is a virtual knot with meridian-longitude pair $\mu, \la \in G_K$ such that $\la$ is nontrivial.
If $K$ admits a $p$-periodic virtual knot diagram $D$, then the corresponding automorphism $\phi \in \Aut(G_K)$ has order $p$ and preserves the peripheral system $P_K$.
In fact, $\phi(\mu) =g \mu g^{-1}$ and $\phi(\la) = g \la g^{-1}$ for some $g \in G_K.$  
\end{proposition}

\begin{proof} Clearly $\phi^p = \id$, so it suffices to show that $\phi^k \neq \id$ for any $1 \leq k < p.$  
We first claim that $\phi \neq \id$. We show this by contradiction. Assuming that $\phi = \id$, it follows that the generators and the relations for $G_K$ in the presentation \eqref{Wirtinger-pres-K} satisfy $a_{i,j} = a_{i,j+1}$ and $r_{i,j} = r_{i,j+1}$ for $1\leq i\leq n$. These observations reveal that if $\phi = \id$, then the quotient map
$\pi\colon G_K \to G_{K_\coAsterisk}$ is an isomorphism of groups.

On the other hand, choose a  meridian $\mu$ and  longitude $\lambda$ for $K$ as follows.
Let $\mu =a_{1,0} \in G_K$, and starting on the arc $a_{1,0}$ and going once around the knot, let $\omega$ be the word obtained by recording the overcrossing arc $(a_{k,\ell})^{\varepsilon_i}$ according to its writhe $\varepsilon_i$ at each undercrossing $c_{i,j}$.
We then set $\lambda = \omega \, (a_{1,0})^{-m}$, where $m$ denotes the exponential sum of $\omega$.
Notice that the homology classes of $\mu, \lambda \in G_K$ can be represented by curves in $T^2 \subset \partial(\Sigma \times I \sm \Int(N(K)))$ 
with $[\mu]\cdot [\lambda] = 1$. 
We can choose a compatible meridian and longitude $\mu^\coAsterisk, \lambda^\coAsterisk \in G_{K_\coAsterisk}$  for $K_\coAsterisk$
such that $\pi(\mu) = \mu^\coAsterisk$ and $\pi(\lambda) =  (\lambda^\coAsterisk)^p.$

However, since $\pi$ is an isomorphism, there is an element $\zeta \in G_K$ with $\pi(\zeta) = \lambda^\coAsterisk$ and $\zeta^p = \lambda$. Thus $\lambda$ admits a $p$-th root. 
In that case, $[\mu]\cdot [\lambda]= p [\mu]\cdot [\zeta]$ is a multiple of $p$, which contradicts the fact that $[\mu]\cdot [\lambda]=1.$

Now suppose $\phi \in \Aut(G_K)$ has order $k$ for some $1<k<p.$
If $k$ is relatively prime to $p,$ then writing $1 = \ell k + m p$, it follows that
$\phi= \id$, which is a contradiction. The only other possibility is when $p=\ell k$ for some $\ell$. In this case, notice that $K$ is $\ell$-periodic, and set $\phi'=\phi^k$. Arguing as before shows that $\phi' \neq \id$, and it follows that $\phi$ cannot have order $k$.
\end{proof}
The hypothesis in Proposition \ref{prop-order} that the longitude is nontrivial is essential.
There are nontrivial virtual knots $K$ with $G_K \cong \ZZ$.
In that case, the longitude is trivial and the automorphism arising from a periodic diagram is also trivial, see Figure \ref{kishino} for a specific example.

\begin{figure}[ht]
\includegraphics[scale=2.0]{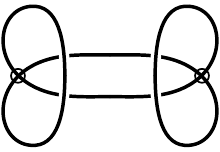}  
\caption{{The Kishino knot (pictured) is 2-periodic and has infinite cyclic knot group.}} \label{kishino}
\end{figure}

\begin{question}
Does there exist a periodic virtual or welded (non-classical) knot with ``nontrivial'' knot group, that is $G_K \not\cong \ZZ$, and trivial longitude for which the induced automorphism $\phi \colon G_K \to G_K$ is the identity?
\end{question}

S.-G.~Kim has shown that any classical knot group is the group of a (non-classical) virtual knot with trivial longitude, see \cite[Corollary 7]{Kim-2000}.

\begin{question}\label{qu:exotic}
Does there exist a periodic virtual or welded knot whose knot group is isomorphic to that of a classical knot such that the induced automorphism
is not conjugate to any automorphism arising from a symmetry of the classical knot?
\end{question}

Note that any automorphism of the knot group of a prime classical knot preserves its peripheral structure, (see \cite[Corollary 4.1]{Tsau}),
and so the knot group in  \hbox{Question \ref{qu:exotic}} 
should be that of a composite classical knot.

\setcounter{section}{2}
\subsection{The group of outer automorphisms}
Given a group $G$,
an {\it inner automorphism} is an element of $\Aut(G)$ of the form $c_g\colon G \rightarrow G$, where $g \in G$ 
and  $c_g(x) =g x g^{-1}$ for all $x \in G$ (conjugation by $g$).
The subset of all inner automorphisms,  $\Inn(G)$, is a normal subgroup of $\Aut(G)$ and the group of {\it outer automorphisms} of $G$ is defined as the quotient
group  $\Out(G) = \Aut(G) / \Inn(G).$
There is an exact sequence
\begin{equation*}
\begin{tikzcd}
1  \arrow{r}{} & Z(G)  \arrow{r}{i} &  G  \arrow{r}{C} & \Inn(G)   \arrow{r}{} & 1, 
\end{tikzcd}
\end{equation*}
where $Z(G) = \{ x \in G ~|~ gx=xg \text{ for all } g \in G\}$  
is the {\it center} of $G$, $i\colon Z(G) \hookrightarrow G$ is inclusion
and $C(g) = c_g$ for $g \in G$.
Hence $g \in Z(G)$ if and only if $c_g = \id_{G}$.

Let $Q  \colon \Aut(G) \rightarrow \Out(G)$ denote the quotient homomorphism.
The following proposition will be used in the proof of Theorem \ref{thm:lift_three}.

\begin{proposition}\label{prop:torfreecenterless}
Let $G$ be a torsion free group with trivial center.
If $\varphi \in \Aut(G)$ has finite order $n \geq 1$, then $Q(\varphi) \in \Out(G)$ also has order $n$.
\end{proposition}

\begin{proof}
Let $m$ be the order of $Q(\varphi)$. 
Note that $m$ divides $n$ and $\varphi^m = c_g$ for some $g \in G$.
Let $s = n/m$. 
Then  $\id_G = \varphi^n = \varphi^{ms} = (c_g)^s = c_{g^s}.$
Since the center of $G$ is trivial it follows that $g^s=1$ and so $g=1$ because $G$ is torsion free.
Hence $\varphi^m =1$ and so $m=n$.
\end{proof}

The following ``relative'' versions of automorphism and outer automorphism groups
are relevant in knot theoretic applications.
Let  $P$ be the conjugacy class of some subgroup $H$ of $G$, that is,
$P = \{g H g^{-1} ~|~ g \in G\}$.
Let $\Aut(G, P)$ be the subgroup of $\Aut(G)$ consisting of all automorphisms that preserve $P$.
Observe that $\Inn(G)$ is a normal subgroup of $\Aut(G, P)$.
Let $\Out(G, P) = \Aut(G, P) / \Inn(G)$.
Note that $\Out(G, P)$ is a subgroup of $\Out(G)$.

\setcounter{section}{3}
\subsection{Automorphisms of knot groups} \label{section:autom}\label{Out-groups}

Given a  knot $K$ in $S^3$,
the inclusion $\partial E_K \hookrightarrow E_K$ together with a choice of basepoint 
induces a homomorphism on fundamental groups that is injective if $K$ is not the unknot.
The conjugacy class of the image of this  homomorphism, which we denote by $P_K$,
is the {\it peripheral system} associated to $K$.
It is known that the pair $(G_K, P_K)$ is a complete knot invariant,
that is, two (classical) knots $K$ and $K'$ are equivalent if and only if the pairs $(G_K, P_K)$ and $(G_{K'}, P_{K'})$ are isomorphic.

A homeomorphism $F \colon E_K \rightarrow E_K$ induces an isomorphism of relative homology groups
$F_* \colon H_3(E_K,\partial E_K) \rightarrow H_3(E_K,\partial E_K)$.
Since $E_K$ is a compact, connected orientable $3$-manifold,  the homology group
$H_3(E_K,\partial E_K)$ is infinite cyclic and so $F_*$
 is either the identity, in which case we say $F_*$ is {\it orientation preserving}, or $-1$ times the identity,
in which case we say $F_*$ is {\it orientation reversing}.
By Alexander Duality, $E_K$ is a homology circle and thus $H_2(E_K) = H_3(E_K) = 0$.
The homology long exact sequence of the pair $(E_K,\partial E_K)$, together with its  naturality, yield a commutative square
\begin{equation*}
\begin{tikzcd}
H_3(E_K,\partial E_K) \arrow[r, "F_*"] \arrow[d,"\delta"']
& H_3(E_K,\partial E_K) \arrow[d, "\delta" ] \\
H_2(\partial E_K) \arrow[r,"{\left( F|_{\partial E_K} \right)_*}"]
& H_2(\partial E_K), \
\end{tikzcd}
\end{equation*}
where the connecting homomorphism $\delta \colon H_3(E_K,\partial E_K) \rightarrow H_2(\partial E_K)$ is an isomorphism.
It follows that $F_*$ is the identity (respectively, $-1$ times the identity) 
if and only if $\left( F|_{\partial E_K} \right)_*$ is the identity (respectively, $-1$ times the identity),
that is, $F$ is orientation preserving (respectively, reversing) if and only if
$F|_{\partial E_K}$ is orientation preserving (respectively, reversing).

Assume that $K$ is not the unknot.
We assign to an automorphism $\varphi \in \Aut(G_K, P_K)$ the status ``orientation preserving'', respectively, ``orientation reversing'', as follows.
Let $H \in P_K$. 
There exists $g \in G_K$ such that $\varphi(H)= g^{-1} H g$.
The composition in group homology
\begin{equation}\label{eq:opreserving}
\begin{tikzcd}
H_2(H)  \arrow{r}{(\varphi|_H)_*} & H_2(g^{-1}Hg) \arrow{r}{{(c_{g})_*}} & H_2(H)
\end{tikzcd}
\end{equation}
is an isomorphism.
This isomorphism is  independent of  the element $g$ satisfying $\varphi(H)= g^{-1} H g$  because if
$g_1^{-1} H g_1 = g_2^{-1} H g_2$, then $u = g_1 g_2^{-1} \in N_{G_K}(H)$,
the normalizer of $H$ in $G_K$, and hence $ u \in H$
since $N_{G_K}(H) = H$ by \cite[Corollary, p.\,148]{Heil-1981}
(this is also a consequence of \cite[Theorem 2]{Simon-1976}).
Observe that  $g_1 = u g_2$ and so $c_{g_1} = c_{ug_2} = c_u c_{g_2}$.
Since $H$ is abelian and $u \in H$, $c_u = \id_H$ and so $c_{g_1}= c_{g_2}$.

Note that  $H_2(H) \cong H_2(\partial E_K)$ is infinite cyclic.
Hence the isomorphism \eqref{eq:opreserving} is either the identity, in which case  we say $\varphi$  is  {\it orientation preserving},
or $-1$ times the identity, in which case  we say $\varphi$  is {\it orientation reversing}.
It is straightforward to show that our definition 
of orientation preserving or reversing for $\varphi$ is independent of the choice of $H \in P_K$.
The set of all orientation preserving automorphisms is a subgroup, denoted $\Aut^+(G_K, P_K)$, of $\Aut(G_K, P_K)$
of index at most $2$.
Observe that any inner automorphism  $\varphi = c_g$ is orientation preserving because $c_{g^{-1}} \varphi = c_{g^{-1}} c_g = \id_H$,
and so orientation preserving and orientation reversing are well-defined notions for outer automorphisms.
We denote the subgroup of orientation preserving outer automorphisms by $\Out^+(G_K, P_K)$.

The geometrically and algebraically defined versions of orientation preserving (or reversing)  are related as follows.
Let $x_0 \in \partial E_K$ be a basepoint and let
\[
H= \operatorname{image}\left(i_*\colon \pi_1(\partial E_K,x_0) \rightarrow \pi_1(E_K,x_0)\right) \subset G_K = \pi_1(E_K, x_0).
\]
If $F \colon E_K \rightarrow E_K$ is a homeomorphism and $\sigma$ is a path in $E_K$  from $x_0$ to $F(x_0)$, then 
the composition
\begin{equation*}
\begin{tikzcd}
\pi_1(E_K,x_0) \arrow{r}{F_*} &  \pi_1(E_K,F(x_0)) \arrow{r}{\sigma^{-1}_{\#}} & \pi_1(E_K,x_0)
\end{tikzcd}
\end{equation*}
is an automorphism $\varphi \in \Aut(G_K, P_K)$, where $P_K$ is the conjugacy class of $H$.
Note that if the path $\sigma$ lies in $\partial E_K$, then $\varphi(H) = H$.
There is a commutative diagram
\begin{equation} \label{eq:hur}
\begin{tikzcd}
\pi_1(\partial E_K, x_0) \arrow[r, "(F|_{\partial E_K})_*"]   \arrow[d,"h"'] & 
\pi_1(\partial E_K, F(x_0)) \arrow[r, "\sigma^{-1}_{\#}" ]     \arrow[d,"h"'] &
\pi_1(\partial E_K, x_0)  \arrow[d,"h"']  & \\
H_1(\partial E_K)  \arrow[r,"(F|_{\partial E_K})_*"] & 
H_1(\partial E_K) \arrow[r, equals] & 
H_1(\partial E_K),
\end{tikzcd}
\end{equation}
where $h$ denotes the Hurewicz homomorphism, which is an isomorphism in the above diagram because $\pi_1(\partial E_K, x_0)$ is abelian.
Since $\partial E_K$ is a $2$-torus,
the Pontryagin product gives an isomorphism
\begin{equation}\label{eq:exterior}
\Lambda^2 H_1(\partial E_K) ~\longrightarrow~ H_2(\partial E_K),
\end{equation}
where $\Lambda^2 H_1(\partial E_K)$ is the second exterior power of $H_1(\partial E_K)$.
It follows from \eqref{eq:opreserving},  \eqref{eq:hur} and \eqref{eq:exterior} that $\phi$ is orientation preserving (respectively, reversing)
if and only if $F$ is  orientation preserving (respectively, reversing);
moreover, $F$  is orientation preserving if and only if 
$\det \left( (F|_{\partial E_K})_* \colon H_1(\partial E_K) \rightarrow H_1(\partial E_K)\right) = 1$,
equivalently, $\varphi$  is orientation preserving if and only if
$\det \left( c_{g} \circ {\varphi|_H} \colon H \rightarrow H \right) = 1$.

\begin{proposition}\label{good-guys-win}
Let $K$ be a classical knot that admits  a $p$-periodic  virtual knot diagram $D$ and let $\varphi \in \Aut(G_K, P_K)$
be the automorphism associated to $D$.
Then $\varphi$ is orientation preserving.
\end{proposition}
\begin{proof}
Let $H$ be the peripheral subgroup of $G_K$ generated by the meridian-longitude pair $\mu, \la \in G_K$ as in Proposition \ref{prop-order}.
Also by Proposition \ref{prop-order},
$\phi(\mu) =g^{-1} \mu g$ and $\phi(\la) = g^{-1} \la g$ for some $g \in G_K$.
Hence
\begin{tikzcd}
H  \arrow{r}{\varphi|_H} &  g^{-1}Hg  \arrow{r}{c_g} & H
\end{tikzcd}
is the identity. 
It follows that the induced homomorphism $(c_g \circ \varphi|_H)_*\colon H_2(H) \rightarrow H_2(H)$ is also the identity and so $\varphi$ is orientation preserving.
\end{proof}

\setcounter{section}{4}
\subsection{Lifting periodic homeomorphisms}

We give a criterion for a periodic homeomorphism of a space to lift to a periodic homeomorphism, of the same period, of its universal cover.

\noindent {\it Convention.}
Let  $\alpha$ and $\beta$ be paths, parametrized by the unit interval, in a space $Z$ such that $\alpha(1) = \beta(0)$.
We denote path multiplication by $\alpha \cdot \beta$.
Moreover, to avoid an excess of notation, for paths $\alpha, \alpha'$
we write $\alpha = \alpha'$ to mean $\alpha$ is homotopic to $\alpha'$ keeping endpoints fixed.
The isomorphism of fundamental groups induced by a path $\alpha$ in  $Z$ is denoted by $\alpha_{\#} \colon \pi_1(Z,\alpha(0)) \rightarrow \pi_1(Z,\alpha(1))$.

Let $X$ be a path connected, locally path connected, semi-locally simply connected space
and let $p\colon Y \rightarrow X$ be a universal covering projection.
Let $x_0 \in X$ be a basepoint and $y_0 \in Y$ a basepoint covering $x_0$, that is, $p(y_0) = x_0$.

\begin{proposition}\label{prop:lift_one}
Let $F \colon X \rightarrow X$ be an $n$-periodic homeomorphism where $n > 1$.
Let $\sigma \colon [0,1] \rightarrow X$ be a path from $x_0$ to $F(x_0)$.
Let ${\widetilde \sigma} \colon [0,1] \rightarrow Y$ be the unique lift of $\sigma$ such that ${\widetilde \sigma}(0) = y_0$
and let ${\widetilde F} \colon Y \rightarrow Y$ be the unique lift of $F$ such that ${\widetilde F} (y_0)  = {\widetilde \sigma}(1)$.
Then ${\widetilde F}$ is an $n$-periodic homeomorphism
if and only if the  element $[\tau] \in \pi_1(X,x_0)$ represented by the loop
$\tau = \sigma \cdot (F \circ \sigma) \cdot (F^2 \circ \sigma) \, \cdots  \, (F^{n-1} \circ \sigma)$
is the identity element.
\end{proposition}

\begin{proof}
 Since $F^n = \id_X$,
 $\big({\widetilde F}\big)^n$ is a lift of  $\id_X$, that is, a covering transformation.
 Furthermore, $\big({\widetilde F}\big)^n(y_0) = {\widetilde \tau}(1)$
 where
 ${\widetilde \tau} = {\widetilde \sigma} \cdot ({\widetilde F} \circ {\widetilde \sigma})  \, \cdots  \,(\big({\widetilde F}\big)^{n-1} \circ {\widetilde \sigma})$.
Note that ${\widetilde \tau}$ is the unique lift of $\tau$ such that ${\widetilde \tau}(0) = y_0$.
Hence $\big({\widetilde F}\big)^n$ is the covering transformation corresponding to the element $[\tau] \in \pi_1(X,x_0)$.
It follows that $\big({\widetilde F}\big)^n = \id_Y$ if and only if $[\tau]$ is the identity element.
 \end{proof}

For  $p\colon Y \rightarrow X$  as above we have the following sufficient condition for lifting an $n$-periodic homeomorphism of $X$ to
an $n$-periodic homeomorphism of $Y$.
 
\begin{theorem}\label{thm:lift_two}
Let $F \colon X \rightarrow X$ be an $n$-periodic homeomorphism where $n > 1$.
Assume that the center of $\pi_1(X,x_0)$ is trivial and that $F$ is homotopic to a map
$f \colon X \rightarrow X$ such that $f(x_0) = x_0$ and
that $f_* \colon \pi_1(X,x_0) \rightarrow \pi_1(X,x_0)$ has order $m$, with $m$ dividing $n$, in $\Aut(\pi_1(X,x_0))$.
Then $F$ lifts to an $n$-periodic homeomorphism ${\widetilde F} \colon Y \rightarrow Y$.
\end{theorem}

\begin{proof}
Let $H\colon f \simeq F$ be a homotopy from $f$ to $F$.  
Let $\sigma(t) = H(x_0, t)$ for $t \in [0,1]$.
Then $\sigma$ is a path from  $x_0$  to $F(x_0)$ and there is a commutative diagram
\begin{equation*}
 \begin{tikzcd}
\pi_1(X,x_0) \arrow[r,"F_*"]
  & \pi_1(X,F(x_0)) \\
\pi_1(X,x_0)  \arrow[u,equal] \arrow[r,"f_*"]
  & \pi_1(X,x_0).  \arrow[u,"\sigma_{\#}"']
\end{tikzcd}
 \end{equation*}
 Let $\tau^{(1)} = \sigma$ and for 
$j=2, \ldots, n$, let $\tau^{(j)}$ be the path from $x_0$ to $F^j(x_0)$ given by
$$\tau^{(j)} = \sigma \cdot (F \circ \sigma)  \, \cdots  \, (F^{j-1} \circ \sigma).$$
For $j=1, \ldots, n$,  we have commutative diagrams
\begin{equation*}
 \begin{tikzcd}
\pi_1(X,F^j(x_0)) \arrow[r,"F_*"]
  & \pi_1(X,F^{j+1}(x_0)) \\
\pi_1(X,x_0)  \arrow[u,"\tau^{(j)}_{\#}"] \arrow[r,"F_*"]
  & \pi_1(X,F(x_0)),  \arrow[u,"(F\circ\tau^{(j)})_{\#}"']
\end{tikzcd}
 \end{equation*}
 and so 
\[
F_*  \tau^{(j)}_{\#} = (F\circ\tau^{(j)})_{\#} F_* = (F\circ\tau^{(j)})_{\#} \sigma_{\#} f_* =  (\sigma \cdot (F\circ\tau^{(j)}))_{\#} f_* =  \tau^{(j+1)}_{\#} f_*.
\]
We have $F_* = \sigma_{\#} f_*  = \tau^{(1)}_{\#} f_*$ and so $F^2_* = F_* \tau^{(1)}_{\#} f_* = \tau^{(2)}_{\#} f^2_*$.
Proceeding inductively, we obtain $F^n_* =  \tau^{(n)}_{\#} f^n_*$.
By hypothesis, $F^n = \id_X$  and $f^n_*$ is the identity.
Hence $\tau^{(n)}_{\#}$ is the identity isomorphism.
Since $\tau = \tau^{(n)}$ is a closed loop, $\tau_{\#}$ is conjugation by $[\tau]$ and
hence $[\tau] \in \pi_1(X,x_0)$ is the identity element because, by hypothesis, $\pi_1(X,x_0)$ has a trivial center.
The conclusion of the theorem follows from Proposition \ref{prop:lift_one}.
\end{proof}
 
\smallskip

We recall some background from \cite[\S3]{Conner-Raymond} that will be needed.
Let $\cE(X)$ be the set of self-homotopy equivalences of $X$.
Give it the compact-open topology and,
in order to avoid pathologies,
also assume in the subsequent discussion that $X$ is a locally compact, separable metric ANR,
for example, a metrizable topological manifold.
Let $\cE(X,x_0)$ be  the subspace  of $\cE(X)$ consisting of  basepoint preserving maps.    
Both $\cE(X)$ and $\cE(X,x_0)$ are H-spaces (with composition of maps as multiplication).
The map  $\Theta \colon \cE(X,x_0) \rightarrow \Aut(\pi_1(X,x_0))$ given by $\Theta(f) = f_*$ induces a group homomorphism
\begin{equation}\label{hequiv-one}
\Theta_* \colon \pi_0(\cE(X,x_0),\id_X) \rightarrow \Aut(\pi_1(X,x_0)).
\end{equation}

Assume that $X$ is path connected.
Given $h \in \cE(X)$ and a path $\sigma$ in $X$ from $x_0$  to $h(x_0)$ the composition
\begin{equation*}
\begin{tikzcd}
\pi_1(X,x_0) \arrow{r}{h_*} &  \pi_1(X,h(x_0)) \arrow{r}{\sigma^{-1}_{\#}} & \pi_1(X,x_0)
\end{tikzcd}
\end{equation*}
yields an automorphism of $\pi_1(X,x_0)$.
Another choice of path from $x_0$  to $h(x_0)$ gives an automorphism that agrees with $\sigma^{-1}_{\#}h_*$ up to an inner automorphism
yielding a well-defined map  $\Psi \colon \cE(X) \rightarrow \Out(\pi_1(X,x_0))$ that induces a group homomorphism
\begin{equation}\label{hequiv-two}
\Psi_* \colon \pi_0(\cE(X),\id_X) \rightarrow \Out(\pi_1(X,x_0)).
\end{equation}
Let $i \colon \cE(X,x_0) \hookrightarrow \cE(X)$ be the inclusion map.
The diagram
\begin{equation}\label{Aut-Out}
\begin{tikzcd}
\pi_0(\cE(X,x_0),\id_X) \arrow[r, "\Theta_*"] \arrow[d,"\pi_0(i)"']
& \Aut(\pi_1(X,x_0)) \arrow[d, "Q" ] \\
\pi_0(\cE(X),\id_X)  \arrow[r,"\Psi_*"]
& \Out(\pi_1(X,x_0))
\end{tikzcd}
\end{equation}
is commutative.

Let $\Homeo(X) \subset \cE(X)$ be the subspace consisting of self-homeomorphisms of $X$.
Note that if $X$ is locally connected and locally compact, then $\Homeo(X)$ with the compact-open topology is a topological group.  
The composition
\begin{equation}\label{Psi-again}
\begin{tikzcd}
\Homeo(X) \arrow[r,hook] &  \cE(X) \arrow{r}{\Psi} & \Out(\pi_1(X,x_0))
\end{tikzcd}
\end{equation}
is a group homomorphism which we also denote by $\Psi$ to avoid notation proliferation.

Recall that a space $X$ is {\it aspherical} if it has a contractible universal covering space.
If $X$ is aspherical, then the homomorphisms  $\Theta_*$  and  $\Psi_*$ in   \eqref{hequiv-one} and \eqref{hequiv-two}, respectively, are both isomorphisms.
For any group $G$, let $G_{\it fin}$ denote the subset of $G$ consisting of elements of finite order
($G_{\! \, \it fin}$ need not be a subgroup).

\begin{theorem}\label{thm:lift_three}
Let $X$ be an aspherical space with universal cover $Y$.
Let $F \colon X \rightarrow X$ be an $n$-periodic homeomorphism where $n > 1$.
Assume that $\pi_1(X,x_0)$ is torsion free and its center is trivial.
If $\Psi(F) \in \operatorname{image}\left(Q \colon \Aut(\pi_1(X,x_0))_{\it fin} \rightarrow \Out(\pi_1(X,x_0))\,\right)$,
then $F$ lifts to an $n$-periodic homeomorphism ${\widetilde F} \colon Y \rightarrow Y$.
\end{theorem}

\begin{proof}
Since $\Psi \colon \Homeo(X) \rightarrow \Out(\pi_1(X,x_0))$ is a group homomorphism,
$\Psi(F)$ has order $m$ dividing $n$.
By hypothesis, there exists $\varphi \in \Aut(\pi_1(X,x_0))$ of finite order such that $\Psi(F) = Q(\varphi)$.
By Proposition \ref{prop:torfreecenterless},
$\varphi$ also has order $m$.
Since $\Theta_*$ and $\Psi_*$ in diagram \eqref{Aut-Out} are isomorphisms (because $X$ is aspherical)
there exists $f \in \cE(X,x_0)$ such that $F$ is homotopic to $f$ and $\Theta(f) = \varphi$.
The conclusion of the Theorem follows from Theorem \ref{thm:lift_two}.
\end{proof}

\begin{remark}\label{def:finitedimCW}
If an aspherical space is  homotopy equivalent to a finite dimensional CW complex, then its fundamental group is torsion free.
\end{remark}

We make use of the following theorem of P.~A.~Smith.

\begin{theorem}[P.~A.~Smith]\label{thm:smith}
A periodic homeomorphism of $\RR^3$ has a fixed point.
\end{theorem}
\begin{proof} 
Let $F\colon \RR^3 \rightarrow \RR^3$ be a periodic homeomorphism.
Regarding $S^3$ as the one-point compactification of $\RR^3$, the map
$F$ extends to a periodic homeomorphism ${\widehat F} \colon S^3 \rightarrow S^3$.
Suppose $F$ has no fixed points.  Then the point at infinity is the only fixed point of  ${\widehat F}$.
By \cite[Theorem 4]{Smith-1939},
the fixed point set of a periodic homeomorphism of $S^3$ (other than the identity) is either empty
or is homeomorphic to one of $S^0$,  $S^1$ or $S^2$.
None of these possible fixed point sets is a singleton, a contradiction.
Hence $F$ must have a fixed point.
\end{proof}

\begin{theorem}\label{thm:fixedpoint}
Let $K \subset S^3$ be a knot that is not a torus knot.
Let $F \colon E_K \rightarrow E_K$ be a periodic homeomorphism.
If $\Psi(F) \in \operatorname{image}\left(Q \colon \Aut(G_K,P_K)_{\it fin} \rightarrow \Out(G_K,P_K)\,\right),$
then $F$ has a fixed point in the interior of $E_K$.
\end{theorem}

\begin{proof}
Since $K$ is not a torus knot, the center of $G_K$ is trivial, \cite{BZ-1966}
(also see \cite[Theorem 4]{Simon-1976}).
Let $M$ be the interior of $E_K$.
Note that the inclusion $M \hookrightarrow  E_K$ is a homotopy equivalence.
Choosing a basepoint $x_0 \in M$, we have $\pi_1(M,x_0) \cong G_K$.
By \cite[Theorem 8.1]{Waldhausen-1968},
the universal covering space of $M$ is homeomorphic to \hbox{$3$-dimensional} 
Euclidean space, $\RR^3$.
Hence $M$ is aspherical and moreover $G_K$ is torsion free (see Remark \ref{def:finitedimCW}).
Let $n$ be the period of $F$.
Observe that $F$ restricts to an $n$-periodic homeomorphism of $M$, which we write as $F' \colon M \rightarrow M$,
and $\Psi(F) = \Psi(F')$.
By Theorem \ref{thm:lift_three}, $F'$ lifts to an $n$-periodic homeomorphism $\widetilde{F'} \colon \RR^3 \rightarrow \RR^3$.
By Theorem \ref{thm:smith},  $\widetilde{F'}$ has a fixed point $y \in \RR^3$. 
The image of $y$ under the covering projection $\RR^3 \rightarrow M$ is a fixed point of $F'$ and hence also of $F$.
\end{proof}

Note that if $F\colon E_K \rightarrow E_K$ is an $n$-periodic homeomorphism with a fixed point $x \in E_K$ and 
$\sigma$ is a path from a basepoint $x_0 \in E_K$  to $x$, then the composite
\begin{equation*}
\begin{tikzcd}
\pi_1(E_K,x_0) \arrow{r}{\sigma_{\#}} &  \pi_1(E_K,x) \arrow{r}{F_*} &  \pi_1(E_K,x) \arrow{r}{\sigma^{-1}_{\#}} & \pi_1(E_K,x_0)
\end{tikzcd}
\end{equation*}
is an automorphism $\varphi \in \Aut(G_K, P_K)$ of order $m$ dividing $n$;
furthermore, if $K$ is not a torus knot, then $m=n$ by the main theorem of \cite{Freedman-Yau}.

\begin{corollary}
Let $K \subset S^3$ be a knot that is not a torus knot.
Let $F \colon E_K \rightarrow E_K$ be a nontrivial periodic homeomorphism.
Then $F$ is a free symmetry if and only if
$$\Psi(F) \notin \operatorname{image}\left(Q \colon \Aut(G_K,P_K)_{\it fin} \rightarrow \Out(G_K,P_K)_{\it fin}\,\right).$$
\end{corollary}

For example, the hyperbolic knot $9_{48}$ is freely periodic with order $p=3$ (see \cite{Hartley-1981} and \cite{Riley-1982}), and so the associated element in $\Out(G_K)$ has order 3 but is not in the image of $\Aut(G_K)_{\it fin}$.

\setcounter{section}{5}
\subsection{Proof of Theorem \ref{thm-main}.}\label{sec:proof}
The proof of  Theorem \ref{thm-main}  divides into two cases, the case $K$ is a torus knot and the case $K$ is not a torus knot.

\medskip \noindent
{\bf Case 1: Torus knots.}  Let  $K$ be the $(r,s)$-torus knot.
The knot group of $K$ admits the well-known presentation:
$$G_K = \langle a, b \mid  a^r = b^s\rangle.$$
In \cite{Schreier-1924}, Schreier determined structure of the  automorphism group of $G_K$ showing that 
$$\Aut(G_K) = \langle I, J, H \mid  J^r = J^s = H^2 = (HI)^2 = (HJ)^2 = I \rangle,$$
where $I$ is conjugation by $a \in G_K$ and $J$ is conjugation by $b \in G_K$  and $H$ is the order two automorphism given on generators by $H(a) = a^{-1}$ and $H(b) = b^{-1}.$

The subgroup of $\Aut(G_K)$ generated by $I$ and $J$ has index two and is isomorphic to the free product $\ZZ/r * \ZZ/s$, so we get a short exact sequence:

$$1\to \ZZ/r * \ZZ/s \to \Aut(G_K) \to \ZZ/2 \to 1.$$

Assume that $\phi \in \Aut(G_K)$ is an automorphism induced from a $p$-periodic
welded diagram of $K$. Proposition \ref{prop-order}
implies that $\phi$ has order $p$.
By Proposition \ref{good-guys-win},
$\varphi$ is orientation preserving and thus $\phi$ lies in the kernel $\Aut(G_K) \to \ZZ/2$.
Thus $\phi$ is an element in $\ZZ/r * \ZZ/s$, and we apply the
Kurosch Subgroup Theorem \cite{Kurosh} to conclude that $p$ must divide $r$ or $s$.
Note that a suitably symmetric picture of an $(r,s)$-torus knot drawn on a torus makes
it evident that such a knot indeed has cyclic periods of order $p$ dividing $r$ or $s$.
This establishes the conclusion of Theorem \ref{thm-main} in the case $K$ is a torus knot.

\medskip \noindent
{\bf Case 2:  Non-Torus knots.} 
Assume that $K$ is not a torus knot.
Since $K$ is not a torus knot, its knot group $G_K$ is torsion free with trivial center (see the proof of Theorem \ref{thm:fixedpoint}).
Thus, if $\phi \in \Aut^+(G_K, P_K)$ has order $p$, then so does its image  $Q(\phi)\in \Out^+(G_K, P_K)$
by Proposition \ref{prop:torfreecenterless}.

For a compact 3-manifold $M$,  let $\operatorname{Diff}(M)$, 
$\operatorname{Homeo}_{PL}(M)$ and $\operatorname{Homeo}(M)$ be, respectively, the groups of diffeomorphisms,  $PL$-homeomorphisms and homeomorphisms of $M$.
It is a consequence of the solution of  the Smale Conjecture by Hatcher  \cite{Hatcher-1983} that the inclusions
$$\operatorname{Diff}(M) \hookrightarrow  \operatorname{Homeo}_{PL}(M) \hookrightarrow \operatorname{Homeo}(M)$$
are homotopy equivalences, and hence induce isomorphisms
$$\Mcg(M) = \pi_0(\operatorname{Diff}(M),\id_{M})  \cong \pi_0(\operatorname{Homeo}_{PL}(M),\id_{M})  \cong \pi_0(\operatorname{Homeo}(M),\id_{M}).$$
As in \eqref{Psi-again},  we denote the composite
\begin{equation*}
\begin{tikzcd}
\operatorname{Homeo}_{PL}(E_K) \arrow[r,hook] &  \cE(E_K) \arrow{r}{\Psi} & \Out(G_K)
\end{tikzcd}
\end{equation*}
by $\Psi$.
Note that for any nontrivial knot $K$ in $S^3$, the knot exterior $E_K$ is a Haken manifold.
Waldhausen's result \cite[Corollary 7.5]{Waldhausen-1968}, which applies to Haken manifolds, asserts that $\Psi$ as above induces an isomorphism
$$ \pi_0(\operatorname{Homeo}_{PL}(E_K),\id_{E_K}) \cong \Out(G_K,P_K).$$
Thus if $Q(\phi)$ has order $p$, then there is a unique $\varphi_* \in \Mcg(E_K)$  (or, equivalently, in $\pi_0(\operatorname{Homeo}_{PL}(E_K),\id_{E_K})$) of order $p$
such that $\Psi_*(\varphi_*) = Q(\varphi)$.

Assume that $K$ admits a $p$-periodic virtual or welded diagram $D$.
As explained in \ref{sec:periodic-virtual},
$D$ gives rise to an element $\phi \in \Aut^+(G_K,P_K)$ which by Proposition \ref{prop-order} has order $p$.
As observed above, $Q(\phi)\in \Out^+(G_K,P_K)$ also has order $p$.
Zimmermann's ``Nielsen Realization Theorem'',
\cite[Satz 0.1]{Zimmermann-1982}
together with  its Addendum, \cite[p.\,358]{Zimmermann-1982} 
(where the case of a $3$-manifold with torus boundary components is discussed, also see \cite{Heil-Tollefson}) asserts that  $\phi_*$ can be realized as a PL homeomorphism $F_\phi \colon E_K \to E_K$, which again has order $p$.
Furthermore, $F_\varphi$ is orientation preserving because $\phi$ is orientation preserving (see the discussion in \ref{Out-groups}).
By \cite[Theorem 2]{Luo-1992}, $F_\phi$ extends to a \linebreak 
$p$-periodic
homeomorphism of $S^3$ that preserves $K$.
This shows that the classical knot $K$ admits a symmetry of order $p$, either as a cyclic period or as a free period.
However, since $Q(\phi)= \Psi(F_\phi)$ is the image of an element in $\Aut(G_K,P_K)$ of finite order, Theorem \ref{thm:fixedpoint} asserts that $F_\phi$ must have a fixed point in the interior of $E_K$. In particular, $F_\phi$ is not a free symmetry, and so $K$ is a $p$-periodic classical knot. This completes the proof of Theorem \ref{thm-main}.\qed

\bigskip

In \cite{Hillman-1984}, Hillman showed that a nontrivial classical link admits only finitely many periods, extending  Flapan's result for classical knots \cite{Flapan-1985}.
We do not know whether the same conclusion is true for the virtual and welded periods of classical links with more than one component.
Theorem \ref{thm-main}, if extended to links, would imply that the answer is no, that is, any virtual or welded period of a link is also a classical period.
Note that an alternative approach would be needed to deal with split links, since the exterior of a split link is not Haken and so Waldhausen's result \cite[Corollary 7.5]{Waldhausen-1968} no longer applies.
Extending Theorem \ref{thm-main} to links remains an interesting and open problem.

\subsection*{Acknowledgements} 
We would like to thank Micah Chrisman, Robin Gaudreau, Matthias Nagel, and an anonymous referee for their helpful comments. 
H. Boden and A. Nicas were supported by  grants from the Natural Sciences and Engineering Research Council of Canada.


\bibliographystyle{amsalpha}

\end{document}